\documentclass[smallextended,numbook,runningheads]{svjour3}     
\smartqed  
\usepackage{graphicx}

\usepackage{amsmath}
\usepackage{mathptmx}      
\journalname{BIT}

\usepackage{enumerate}
\usepackage{marginnote}
\usepackage{algorithm}
\usepackage{amsfonts}
\usepackage{algorithm,algcompatible,amsmath}
\algnewcommand\INPUT{\item[\textbf{Input:}]}%
\algnewcommand\OUTPUT{\item[\textbf{Output:}]}%
\usepackage{amsmath, amsthm, amssymb}
\newtheorem{prop}{Proposition}

\usepackage[utf8]{inputenc} 
\usepackage[T1]{fontenc}    
\usepackage{booktabs}       
\usepackage{amsfonts}       
\usepackage{nicefrac}       
\usepackage{microtype}      
\usepackage{amsopn}
\usepackage{color}
\usepackage{soul}

\usepackage{mathtools}

\definecolor{orange}{rgb}{1,0.5,0}

\newcommand{\ts}{\thinspace}

\usepackage{subcaption}
\newcommand{\modb}[1]{\textcolor{black}{#1}\index {#1}}

\newcommand{\moda}[1]{\textcolor{black}{#1}\index {#1}}
\newcommand{\modg}[1]{#1}
\newcommand{\modgg}[1]{\textcolor{black}{#1}\index {#1}}
\newcommand{\modkk}[1]{\textcolor{black}{#1}\index {#1}}

\begin{document}

\title{Explicit Stabilised Gradient Descent for  Faster Strongly Convex Optimisation}


\author{Armin Eftekhari          \and  Bart Vandereycken  \\ Gilles Vilmart  \and  Konstantinos C. Zygalakis\thanks{The authors are ordered alphabetically.}
}


\institute{
		A.\ Eftekhari\at
              Department of Mathematics and Mathematical Statistics,  \\ Umea Universitet,  \\
              901 87 Umea Sweden
              \email{armin.eftekhari@umu.se}           
           \and
           B.\ Vandereycken \at
              \modg{Section of Mathematics\\
              University of Geneva\\
              CP 64, 1211 Geneva 4,} Switzerland\\
             \email{bart.vandereycken@unige.ch}
             \and 
             G.\ Vilmart \at 
              \modg{Section of Mathematics\\
              University of Geneva\\
              CP 64, 1211 Geneva 4,} Switzerland\\
              \email{gilles.vilmart@unige.ch}
              \and 
              K.\ C.\ Zygalakis\at
              School of Mathematics\\
              University of Edinburgh\\
              Edinburgh EH9 3FD, UK\\
              \email{k.zygalakis@ed.ac.uk}
              }


\maketitle

\begin{abstract}
This paper introduces the \modgg{Runge-Kutta Chebyshev descent method (RKCD)} for strongly
convex optimisation problems. This new algorithm is based on explicit stabilised
integrators for stiff differential equations, a powerful class of numerical schemes that 
avoid the severe step size restriction faced by standard explicit integrators. For optimising quadratic and strongly convex functions, this paper  proves that RKCD nearly
achieves the optimal convergence rate of the conjugate gradient algorithm, and the 
suboptimality of RKCD diminishes as the condition number of the quadratic function 
worsens. It is established that this optimal rate is obtained also for a {partitioned}
variant of 
RKCD applied to perturbations of quadratic functions. In addition, numerical 
experiments on general strongly convex problems show that RKCD outperforms Nesterov's accelerated gradient descent.
\keywords{Runge-Kutta methods \and Strongly Convex Optimization \and Accelerated Gradient Descent}
\subclass{90C25 \and 65L20}
\end{abstract}

\section{Introduction}
Optimisation is at the heart of many applied mathematical 
and statistical problems, while its beauty lies in the simplicity of 
describing the problem in question. 
In this work, given a function $f\colon \mathbb{R}^{d} \to \mathbb{R}$, {we are} 
interested in finding a minimiser
 $x_{*} \in \mathbb{R}^d$ of the \moda{problem} 
\begin{equation}
\min_{x\in \mathbb{R}^d}f(x).
\label{eq:main}
\end{equation} 
We make the common assumption throughout that $f \in \mathcal{F}_{\ell,L}$, namely, the set of $\ell$-strongly convex differentiable functions that have $L$-Lipschitz continuous derivative~\cite{N14}. 
{Corresponding to $f$ is its \emph{gradient flow}, defined as 
\begin{equation} \label{eq:gradient}
\frac{dx}{dt}=- \nabla f(x), \quad  x(0)=x_{0}\in \mathbb{R}^d,
\end{equation}
where $x_0$ is its initialisation. It is easy to see that traversing the gradient flow always reduces the value of $f$. Indeed, for any positive $h$, it holds that 
\begin{equation}
f(x(h)) - f(x_0) = -\int^{h}_{0} \|\nabla f(x(t))\|_2^2 \,dt \modg{\leq 0}.
\end{equation}
By discretising the gradient flow in~\eqref{eq:gradient}, we can design various optimisation algorithms for~\eqref{eq:main}. For example, by substituting \modg{in~\eqref{eq:gradient}} the approximation 
\begin{equation}
\label{eq:app1}
\frac{dx}{dt}\modg{(t_n)} \approx \frac{x(t_n+h)-x(t_n)}{h}, 
\end{equation}
we obtain the gradient descent (GD) method as the iteration
\begin{equation} \label{eq:Euler}
x_{n+1} = x_n - h   \nabla f(x_n), \quad n=0,1,2,\ldots
\end{equation} 
Here, $x_n$ is the numerical approximation of $x(t_n)$ for all $n$ and $h>0$ is the step size~\cite{N14}. For this discretisation to remain \emph{stable}, that is, for $x_n$ in~\eqref{eq:Euler} to remain close to the exact gradient flow $x(t_n)$ and\moda{, consequently,} for the value of $f$ to reduce in every iteration, the step size $h$ must not be too large.

Indeed, a well-known shortcoming of GD is that we must take $h \le 2/L$  to ensure stability, otherwise $f$ might increase from one iteration to the next \cite{N14}. \modkk{One can consider a different discretization of \eqref{eq:gradient}, by for example substituting in~\eqref{eq:gradient}  the approximation~\eqref{eq:app1} at $x(t_{n}+h)$ instead of $x(t_n)$. We then arrive at the update}
\begin{equation}
x_{n+1} = x_n -h  \nabla f(x_{n+1}),
\label{eq:implicit}
\end{equation}
which is known as the \emph{implicit Euler method} in numerical analysis 
\cite{HaW96} because, as the name suggests, it involves solving~\eqref{eq:implicit} for 
$x_{n+1}$. It is not difficult to see that, unlike GD, there is no \modg{size restriction} on the step size $h$ 
for the implicit Euler method \modg{to decay}, a property 
\modg{related to its algebraic stability~\cite{HaL14}}. Moreover, it is  easy to verify that $x_{n+1}$ in 
\eqref{eq:implicit} is  also the unique minimiser of the \moda{problem}
\begin{equation}
\min_{x\in \mathbb{R}^d} \, h f(x)+ \frac{1}{2}\|  x - x_{n}\|_2^2;
\end{equation}
the map from $x_n$ to $x_{n+1}$ is known in the optimisation literature as the \emph{proximal map} of the function $h f$ \cite{NB14}. Unfortunately, even if $\nabla f$ is known explicitly, solving~\eqref{eq:implicit} for $x_{n+1}$ or equivalently computing the proximal map is often just as hard as solving  ~\eqref{eq:main}, with  a few notable exceptions \cite{NB14}. This setback  severely limits the applicability of the  proximal algorithm in~\eqref{eq:implicit} for solving  \moda{problem}~\eqref{eq:main}.

\vspace {-0.3cm}
\paragraph{Contributions.}
With this motivation, we propose the \modgg{\emph{Runge-Kutta Chebyshev descent} (RKCD)} method  for solving  \moda{problem}~\eqref{eq:main}. RKCD offers the best of both worlds, namely the computational tractability of GD (explicit Euler method) and the stability of the proximal algorithm (implicit Euler method). Inspired by \cite{SRB17},  {RKCD} uses \emph{explicit stabilised methods} \cite{SSV98,AbM01,Abd02} to discretise the gradient flow (\ref{eq:gradient}). 

\modg{For the numerical integration of stiff problems,}
explicit stabilised methods provide a computationally efficient alternative to the implicit Euler method for stiff differential equations, where standard integrators face a severe step size restriction, in particular 
for spatial discretisations of \moda{high-dimensional} diffusion PDEs; see the review~\cite{Abd11}.  Every iteration of RKCD consists of $s$ \emph{internal stages}, where each stage performs a simple GD-like update. 
Unlike GD however, RKCD does \emph{not} decay monotonically along its internal stages, 
which allows it to take longer steps and travel faster along the gradient flow. 
After $s$ internal stages, RKCD ensures that its new iterate is stable, namely, the value of $f$ indeed decreases after each iteration of RKCD. 

\modkk{Recently, there has been a revived interest about the design and the interpretation  of optimization methods as discretizations of ODEs \cite{SRB17}. In particular, discrete gradient methods were used in \cite{ERR18} for the integration of \eqref{eq:gradient}  and shown to have similar properties to the gradient descent for (strongly) convex objective functions. In addition, the work in  \cite{WMW19}, considers numerical discretizations of a rescaled version of the gradient flow \eqref{eq:gradient} and shows that acceleration can be achieved when extra smoothness assumptions are imposed to the objective function $f$. Furthermore, in \cite{SBC16,WWJ16} an alternative second-order differential equation to the gradient flow was introduced containing a momentum term. Similarly, to the spirit of this work,  a number of different numerical discretizations including Runge-Kutta methods were used for the integration of this second-order equation and shown to  behave in an accelerated manner \cite{zhang2018direct,shi2019acceleration,BJW18}. 
Our method, on the other hand, can achieve similar acceleration by a direct integration of the gradient flow~\eqref{eq:gradient} and does not need to include such a momentum term explicitly. }

The rest of this paper is organised as follows. Section~\ref{sec:RKCD} formally introduces RKCD, which is summarised in Algorithm \ref{alg:RKCD} and accessible without reading the rest of this paper. In Section \ref{sec:quadratic prog}, we then quantify the performance of RKCD for solving strongly convex quadratic programs, while in Section~\ref{sec:PRKCD} we introduce and study theoretically a 
\moda{composite} 
variant of RKCD applied to perturbations of quadratic functions.} Then in Section~\ref{sec:numerics}, we empirically compare RKCD to other first-order optimisation algorithms and conclude that RKCD improves over the state of the art in practice. This paper concludes with an overview of the remaining theoretical challenges.

\section{Explicit stabilised gradient descent \label{sec:RKCD}}
{Let us start with the simple scalar \moda{problem} where $f(x) = \tfrac{1}{2}\lambda x^2$, that is,
\begin{equation}
\min_{x\in\mathbb{R}} \,\tfrac{1}{2}\lambda x^2, \qquad \lambda > 0,
\label{eq:quad scalar}
\end{equation}
and consider the corresponding gradient flow
\begin{equation}
\frac{dx}{dt}=-\lambda x,\qquad x(0)=x_0\in \mathbb{R}, 
\label{eq:test}
\end{equation}
also known as the \emph{Dahlquist test equation} \cite{HaW96}.
It is obvious from~\eqref{eq:test} that $\lim_{t \rightarrow 
\infty} x(t)=0$ and any sensible optimisation algorithm should provide iterates {$\{x_n\}_n$} with a similar property, that is,
\begin{equation}
\lim_{n \to \infty} x_n = 0.
\label{eq:good lim}
\end{equation}
For GD, which corresponds to the explicit Euler disctretisation of~\eqref{eq:test}, it easily follows from~\eqref{eq:Euler} that 
\begin{equation}
x_{n+1} = R_{gd}(z) x_n, 
\quad 
R_{gd}(z) = 1+z,\quad z = -\lambda h,
\label{eq:GD stab fcn}
\end{equation}
where $R_{gd}$ is the \emph{stability polynomial} of GD. Hence,~\eqref{eq:good lim} holds if $z\in \mathcal{D}_{gd}$, where the \emph{stability domain} $\mathcal{D}_{gd}$ of GD is defined as 
\begin{equation} \label{eq:defS}
\mathcal{D}_{gd} = \{z\in \mathbb{C}\ ;\ |R_{gd}(z)| < 1\}. 
\end{equation}
That is,~\eqref{eq:good lim} holds if $h \in (0,2/\lambda)$, which imposes a severe limit on the time step $h$ when $\lambda$ is large. Beyond this limit, namely, for larger step sizes, the iterates of GD might not necessarily reduce the value of $f$ or, put differently, the explicit Euler method might no longer be a faithful discretisation of the gradient flow. 

At the other extreme, for the proximal algorithm which corresponds to the implicit Euler discretisation of the gradient flow in~\eqref{eq:test}, it follows from~\eqref{eq:implicit} that 
\begin{equation}
x_{n+1} = R_{pa}(z) x_n,
\quad 
R_{pa}(z) = \frac{1}{1-z},
\quad  z = -\lambda h,
\end{equation}
with the stability domain 
$$
\mathcal{D}_{pa} = \left\{ z\in \mathcal{C}\ ; |R_{pa}(z)|<1 \right\}. 
$$
Therefore,~\eqref{eq:good lim} holds for \emph{any} positive step size $h$. This property is known as A-stability of a numerical method~\cite{HaW96}. Unfortunately{, the proximal algorithm} (implicit Euler method) is often computationally intractable, particularly in higher dimensions. 

In numerical analysis, explicit 
stabilised methods for discretising the gradient flow 
offer the best of both worlds, as they are not only explicit and thus computationally tractable, but  they also share some favourable stability 
properties of the implicit 
method. Our main contribution in this work is adapting these methods for optimisation, as detailed next. 



{For discretising the gradient flow~\eqref{eq:test}, the key idea behind explicit stabilised methods is to relax the requirement that every step of explicit Euler method should remain stable, namely, faithful to the gradient flow. This relaxation in turn allows the explicit stabilised method to take longer steps and traverse the gradient flow faster. 
To be specific, applying any given explicit \emph{Runge-Kutta method} with $s$ stages (\moda{i.e.,} 
$s$ evaluations of $\nabla f$) per step to~\eqref{eq:test} yields a recursion of the form
\begin{equation}
x_{n+1} = R_s(z) x_n,
\quad R_s (z) = 1+z + a_2 z^2 + \ldots + a_s z^s, 
\label{eq:explicit stabilised}
\end{equation}
with the corresponding stability domain 
$
\mathcal{D}_s = \left\{ z\in \mathbb{C}\ ; |R_s(z)| < 1 \right\}
$.
We wish to choose $\{a_j\}_{j=2}^s$ to maximise the step size $h$ while ensuring that $z=-h\lambda$ still remains in the stability domain $\mathcal{
D}_{s}$, namely, for the update of the explicit stabilised method to remain stable. More formally, we wish to solve 
\begin{equation}
\displaystyle \max_{a_2,\cdots, a_s} L_s 
\quad \text{subject to}\quad  |R_s(z)| \le 1,\quad   \forall z\in [-L_s,0]. 
\label{eq:stab_polynomial}
\end{equation}}

As shown in \modg{\cite{HaW96} (see also~\cite{Abd11})}, the solution to~\eqref{eq:stab_polynomial} is $L_{s}=2s^{2}$ and, after substituting the optimal values for $\{a_2\}_{j=1}^s$ in~\eqref{eq:explicit stabilised}, we find that the {unique} corresponding $R_s(z)$ is the shifted Chebyshev polynomial $R_{s}(z)=T_{s} (1+z /s^{2})$ where $T_s(\cos \theta)=\cos(s\theta)$ is the Chebyshev polynomial of the first kind with {degree} $s$. In Figure~\ref{fig:stabdetplot}, $R_{s}(z)$ is depicted as $\eta=0$ in red. It is clear from panel (b) that $R_{s}(z)$ equi-oscillates between $-1$ and $1$ on $z \in [-L_s,0]$, which is a typical property of minimax polynomials. As a consequence, after every $s$ internal stages, the new iterate of the explicit stabilised method remains stable and faithful to~\eqref{eq:test}, while travelling the most along the gradient flow.

{Numerical stability is still an issue for the explicit stabilised method outlined above, particularly for the values of $z=-\lambda h$ for which $|R_s(z)|=1$. As seen on the top of Figure \ts\ref{fig:stabdetplot}(a), even the slightest numerical imperfection due to round-off will land us outside of the stability domain $\mathcal{D}_{s}$, which might make the algorithm unstable. }
In addition, {for such values of $z$ \moda{where $|R_s(z)|=1$}, the new iterate $x_{n+1}$ is not necessarily any closer to the minimiser, here the origin.} 
\modb{As a solution, it is common practice (see, e.g., \cite{HaW96}) to 
tighten the stability requirement  to  $|R_s(z)| \leq \alpha_s(\eta) <1$ for every $z \in [-L_{s,\eta},-\delta_{\eta}]$. A popular choice is to introduce a positive dampening parameter $\eta$ so that the stability function satisfies} 
\begin{equation} \label{eq:exp_stab}
R_{s}(z)=\frac{T_{s}(\omega_{0}+\omega_{1}z)}{T_{s}(\omega_{0})}
, \quad \omega_{0}=1+\frac{\eta}{s^{2}}, \quad \omega_{1}
=\frac{T_{s}(\omega_{0})}{T'_{s}(\omega_{0})}.
\end{equation}
\moda{Now, $R_s(z)$ oscillates  between $-\alpha_s(\eta) $ and $ \alpha_s(\eta) $ for every $z\in [-L_{s,\eta},-\delta_\eta]$, where
\begin{equation}\label{eq:alpha}
\alpha_s(\eta) = \frac1{T_s(\omega_0)} <1,
\qquad 
L_{s,\eta} = \frac{1+\omega_0}{\omega_1}.
\end{equation}
In fact, $L_{s,\eta} \simeq  (2-\frac{4}3 \eta)s^2$ is close to the optimal stability domain size $L_{s,0}=2s^2$} for small $\eta$; see 
\cite{Abd11}.}
It also follows from~\eqref{eq:explicit stabilised} that 
\[
|x_{n+1}| \leq \alpha_s(\eta) |x_n|,
\]
namely, the new iterate of the explicit stabilised method is indeed closer to the minimiser. 
In addition, as we can see in Figure \ref{fig:stabdetplot},  introducing damping also ensures that a strip around the negative real axis is included in the complex stability domain $\mathcal{D}_s$, which grows in the imaginary direction as the damping parameter $\eta$ increases.
\moda{We also point out that, }while a small damping $\eta=0.05$ is usually sufficient for standard stiff problems, the benefit of large damping $\eta$ was first exploited in~\cite{AbC08} in the context of stiff stochastic differential equations and later improved in~\cite{AAV18} using second kind Chebyshev polynomials.

 \begin{figure}

    \begin{subfigure}[t]{0.95\textwidth}
      \centering
      \includegraphics[width=0.95\linewidth]{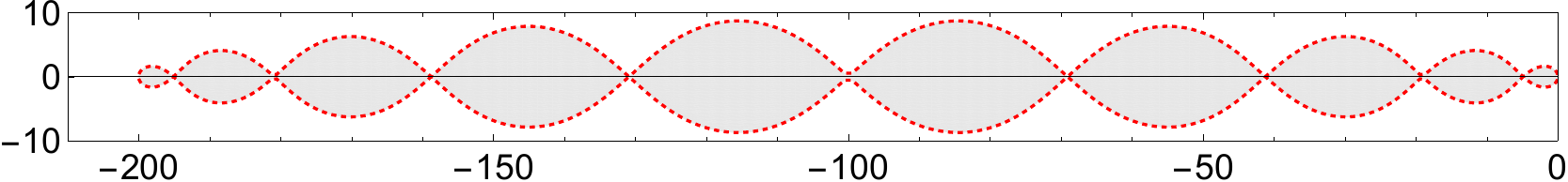}\\[2.ex]
      \includegraphics[width=0.95\linewidth]{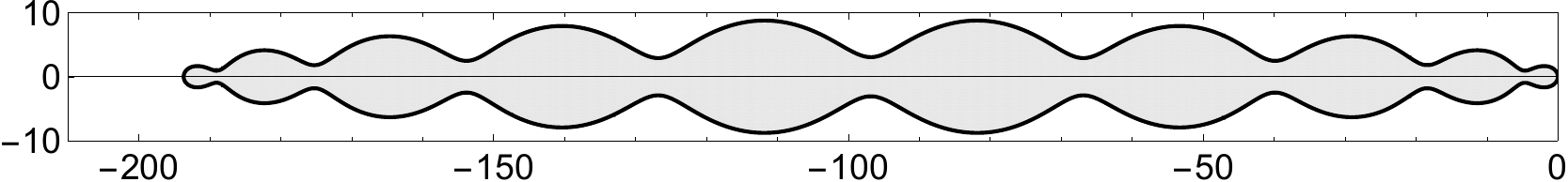}\\[2.ex]
      \includegraphics[width=0.95\linewidth]{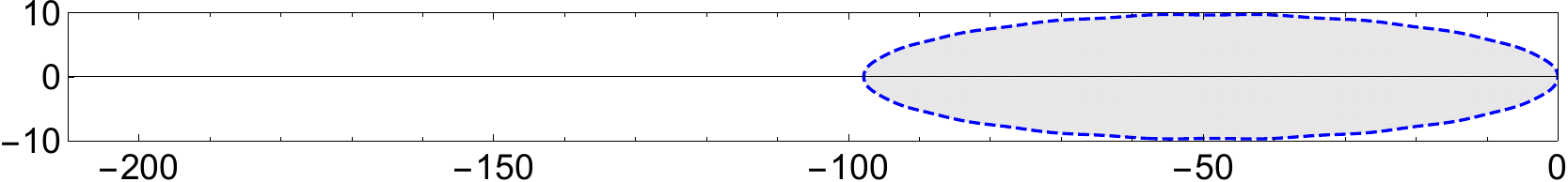}
      \caption{Complex stability domains $\mathcal{D}_{s}$: level set of $|R_s(z)| \leq 1$ for complex $z$.}
    \end{subfigure}\\[2ex]
    \begin{subfigure}[t]{0.95\textwidth}
      \centering
      \includegraphics[width=0.95\linewidth]{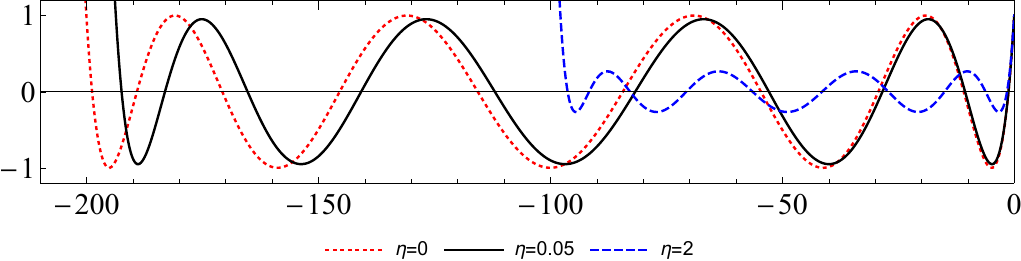}
      \caption{Graph of the stability function $R_{s}(z)$ as function of real $z$.}
    \end{subfigure}
    \caption{
      Stability domains and stability functions of the Chebyshev method with $s=10$ stages and different damping values $\eta=0, 0.05$, and $2$.
      \label{fig:stabdetplot}}
  \end{figure}

It is relatively {straightforward} to generalise the ideas from above to integrate general multivariable gradient flows, thus not only the scalar test function~\eqref{eq:test}.  {For the algorithmic implementation to be numerically stable}, however, one \moda{should} not evaluate the Chebyshev polynomials $T_s(z)$ {{naively} but instead} use their well-known three-term recurrence \modgg{$T_{j+1}(x)=2xT_j(x) - T_{j-1}(x)$, for $j=1,2,\ldots$ (with $T_0(x)=1$) in the implementation of the methods} \cite{VaS80}. We dub this algorithm \modgg{the Runge-Kutta Chebyshev descent method (RKCD), summarised in Algorithm \ref{alg:RKCD}.}

\begin{algorithm}
\caption{\modgg{Runge-Kutta Chebyshev descent method (RKCD)} for solving  ~\eqref{eq:main} \label{alg:RKCD}}

\textbf{Input:}
The gradient $\nabla f$ of a differentiable function $f:\mathbb{R}^d\rightarrow\mathbb{R}$.
Damping $\eta$ (e.g., $\eta=1.17$). Lower and upper bounds $\ell$ and $L$ for eigenvalues of $\nabla^2 f$.  Initialisation $x_0\in\mathbb{R}^d$. 

\vspace{0.1cm}

\textbf{Body:} 
Until convergence, \textbf{repeat}:
\begin{itemize}
  \item Compute $h$ and $s$ using~\eqref{eq:condition_2a} and
\begin{equation}\label{eq:defn of parameters}
\omega_{0}=1+\frac{\eta}{s^{2}}, \quad \omega_{1}
=\frac{T_{s}(\omega_{0})}{T'_{s}(\omega_{0})},
\end{equation}
with $T_s$ is the Chebyshev polynomial of the first kind with degree $s$. 
\item Set $x_n^0 = x_n$ and $x_n^1 = x_n^0- h\mu_1 \nabla f(x_n)$ with $\mu_1 = \omega_1/ \omega_0$
\item For $j\in \{2,\cdots,s\}$, \textbf{repeat}:
$$x_n^j = -\mu_j h \nabla f(x_n^{j-1}) + \nu_j x_{n}^{j-1}-(\nu_j-1) x_{n}^{j-2},$$ 
\begin{equation*}
\mu_j = \frac{2\omega_1 T_{j-1}(\omega_0)}{T_j(\omega_0)},
\quad \nu_j = \frac{2\omega_0 T_{j-1}(\omega_0)}{T_j(\omega_0)}.
\end{equation*}
\item Set $x_{n+1}=x_{n}^s$. 
\end{itemize}

\vspace{0.1cm}

\textbf{Output: } 
Estimate $\widehat{x}={x}_{n+1}$ of a minimiser of  ~\eqref{eq:main}. 

\end{algorithm}

\section{{Strongly convex quadratic } \label{sec:quadratic prog}}

Consider the \moda{problem}  ~\eqref{eq:main} with 
\begin{equation} \label{eq:quad}
f(x)=\tfrac{1}{2}x^{T}Ax-b^{T}x,
\end{equation}
where $A \in \mathbb{R}^{d \times d}$ 
is a positive definite and symmetric matrix, and $b \in \mathbb{R}^{d}$. As the next \modg{proposition} shows, the convergence of \moda{any} Runge-Kutta method (such as GD, proximal algorithm)  depends on the eigenvalues of $A$; the proof can be found in the appendix.
\begin{prop} \label{prp:basic}
For solving \moda{problem}  (\ref{eq:main}) with $f$ as in~(\ref{eq:quad}), consider an optimisation algorithm with stability function $R$ (for example, $R=R_{gd}$ for GD \moda{in (\ref{eq:GD stab fcn})}) \moda{and} step size $h$. Let $\{x_n\}_{n\ge 0}$ be the iterates of this algorithm.  Also \moda{assume that}  $0 <  \ell = \lambda_{1} \leq \cdots \leq \lambda_{d} = L$, \moda{where $\{\lambda_{i}\}_i$ are} the eigenvalues of $A$. Then, for every \moda{iteration} $n\ge 0$, it holds
\begin{equation}\label{eq:decay_quad}
f(x_{n+1})-f(x_{*}) \leq   \max_{1 \leq i \leq d}R^{2}(-h \lambda_{i}) \, ( f( x_{n})- f(x_{*}))
\end{equation}  
with $x_*$ the minimizer of $f$.
\end{prop}
\noindent Let us next apply Proposition \ref{prp:basic} to both GD and RKCD.


\paragraph{Gradient descent. }
\moda{Recalling~\eqref{eq:GD stab fcn}}, it is not difficult to verify that 
\begin{equation} \label{eq:max_R}
\max_{1 \leq i \leq d}R_{gd}^{2}(-\lambda_{i}h)  =
\begin{cases}
R_{gd}^{2}(-\ell h), \quad \text{if} \quad 0< h \leq 
\frac{2}{\ell+L}, \\
\\
R_{gd}^{2}(- L h ), \quad \text{if} \quad   h  \geq  
\frac{2}{\ell+L}.
 \end{cases}
\end{equation}
It \moda{then} follows from~\eqref{eq:max_R} that we must take $ h \in (0,2/L)$ for GD to be stable, namely, for GD to reduce the value of $f$ \moda{in every iteration}.
In addition, one obtains the best possible decay \moda{rate}
by choosing $h 
=2/(\ell+L)$. More specifically,  we have that 
\begin{align*}
\min_{h} \max_{1 \leq i \leq d} R_{gd}^2(-\lambda_i h)   
 = \left(\frac{\kappa-1}{\kappa+1} \right)^2,
\end{align*}
where $\kappa = L/\ell$ is the \emph{condition number} of \moda{the} matrix $A$.
That is, the best convergence rate for GD is predicted by Proposition \ref{prp:basic} as 
\begin{equation}
f(x_{n+1})-f(x^*) \le \left(\frac{\kappa-1}{\kappa+1} \right)^2 (f(x_{n})-f(x^*)),
\end{equation}
achieved with the step size $h=2/(\ell+L)$. Remarkably, the same conclusion holds 
for any function in $ \mathcal{F}_{\ell,L}$, as discussed in  \cite{LRP16}.

\paragraph{Explicit stabilised gradient descent.}
In the case of Algorithm \ref{alg:RKCD}, there are three different parameters that need to be chosen, namely, the step size $h$, the number of internal stages $s$, and the damping factor $\eta$. For \moda{a} fixed positive $\eta$, let us next select $h$ and $s$ so that the numerator of \moda{$R_s(z)$ in}~\eqref{eq:exp_stab}, namely $|T_s(\omega_0+\omega_1 z)|$, is bounded by one. Equivalently, we take $h,s$ such that 
\begin{equation} \label{eq:stabconst}
-1 \le \omega_0 - \omega_1 L h  \le \omega_0 - \omega_1 \ell h \leq 1.
\end{equation}
For an efficient algorithm, we choose the smallest step size $h$ and the smallest number $s$ of internal stages such that 
\eqref{eq:stabconst} holds. 
More specifically,~\eqref{eq:stabconst} dictates that  $\kappa \leq (1+\omega_0)/(-1+\omega_0) = 1+2s^2/\eta$; see (\ref{eq:defn of parameters}). 
This in turn determines the parameters $s$ and $h$ as
\begin{equation} \label{eq:condition_2a}
s =\left\lceil \sqrt{{(\kappa-1)\eta}/2} \right \rceil, 
\qquad h = \frac{\omega_0-1}{\omega_1 \ell},
\end{equation}
with $\kappa = L/\ell$. 
Under~\eqref{eq:stabconst} and using the definitions in (\ref{eq:exp_stab},\ref{eq:alpha}), we find  that 
$$
|R_{rkcd}(-\lambda_i h )|\leq \alpha_s(\eta)
$$
for every $1 \leq i \leq d$.  Then, an immediate consequence of Proposition \ref{prp:basic} is 
\begin{equation}\label{eq:decayf}
f(x_{n+1})-f(x_{*}) \leq  \alpha_s(\eta)^2 ( f( x_{n})- f(x_{*})).
\end{equation}
 %

\noindent Given that every iteration of RKCD consists of $s$ internal stages---{hence, with the same cost as $s$ GD steps---}it is natural to define the \emph{effective} convergence rate of RKCD as 
\begin{equation} \label{eq:rate}
c_{rkcd}(\kappa) =\alpha_{s}(\eta)^{2/s}.
\end{equation}
The following result, proved in the appendix, evaluates the effective convergence rate of RKCD, as given by~\eqref{eq:rate}, \moda{in} the limit of $\eta \rightarrow \infty$.
\begin{prop} \label{prp:eff}
With the choice of step size $h$ and number $s$ of internal stages  in (\ref{eq:condition_2a}), the effective convergence rate $c_{rkcd}(\kappa)$ of RKCD for solving  \moda{problem}~(\ref{eq:main}) with $f$  given in (\ref{eq:quad}) satisfies 
\begin{eqnarray} \label{eq:effective_rate}
\lim_{\eta \rightarrow \infty} c_{rkcd}(\kappa) &=& c_{opt}(\kappa)+O( \kappa^{-3/2} ), \nonumber \\
 c_{opt}(\kappa) &=& \left(\frac{\sqrt{\kappa}-1}{\sqrt{\kappa}+1}\right)^2.
\end{eqnarray}
\end{prop}
\noindent Above, $c_{opt}(\kappa)$ is the optimal convergence rate of a first-order algorithm, which is achieved by the conjugate gradient (CG) algorithm for quadratic $f$; see~\cite{N14}. Put differently, Proposition \ref{prp:eff} states that RKCD nearly achieves the optimal convergence rate in the limit of $\eta\rightarrow\infty$\footnote{In practice,  as $\eta$ becomes larger the number of stages $s$ grows; see (\ref{eq:condition_2a}). \moda{Therefore, the largest possible} $\eta$ \moda{is dictated by} the computational budget in terms of gradient evaluations.}.  
\moda{Moreover, it is  perhaps remarkable that}
the performance of RKCD relative to the conjugate gradient \emph{improves} as the condition number of $f$ worsens, namely, as $\kappa$ increases. The non-asymptotic behaviour of RKCD is numerically investigated in Figure \ref{fig:figc1}, corroborating Proposition \ref{prp:eff}.
As illustrated in Section~\ref{sec:numerics}, Algorithm~\ref{alg:RKCD} \moda{can also be used effectively for optimization problems with non-quadratic $f \in \mathcal{F}_{\ell,L}$}.\footnote{We remark that the parameters $h$ and $s$ should then also be chosen using~\eqref{eq:condition_2a}, where $\kappa=L/\ell$.}
\begin{figure}[tb]
\hspace{-1cm}
		\includegraphics[width=1.1\linewidth]{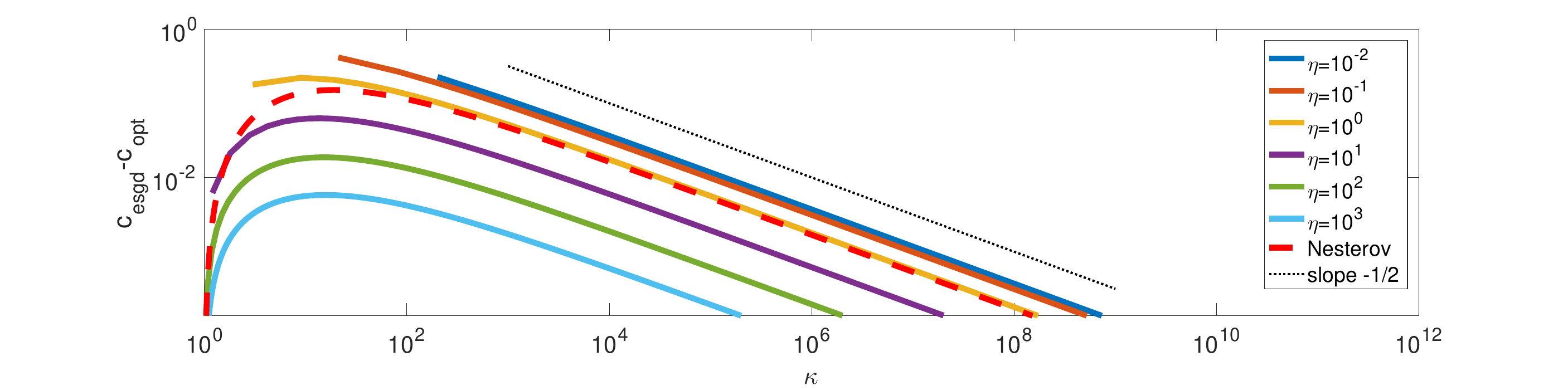}
		\caption{This figure considers the non-asymptotic scenario from Prop.~\ref{prp:eff} and plots, as a function of the condition number $\kappa$,  the difference $c_{rkcd}(\kappa)-c_{opt}(\kappa)$ of the (effective) convergence rate of RKCD compared to the optimal one, for many fixed values of the damping parameter $\eta$. 
In this plot,	we observe that $c_{rkcd}(\kappa) - c_{opt}(\kappa)$ decays as $C_\eta/\sqrt{\kappa}$ for $\kappa\rightarrow \infty$ and for fixed $\eta$, see the slope $-1/2$ in the plot.
Numerical evaluations suggest that the constant $C_\eta=O(1/\sqrt{\eta})$ becomes arbitrarily small as $\eta$ grows, which corroborates Proposition \ref{prp:eff}. %
	For comparison, we also include the result $c_{agd}(\kappa)-c_{opt}(\kappa)$ for the optimal rate of the Nestrov's accelerated gradient descent (AGD) given by
	$c_{agd} = (1-2/\sqrt{3\kappa+1})^2$
	\protect \cite{LRP16}.
Comparing the two algorithms in Section \ref{sec:numerics} we find that RKCD outperforms AGD, namely $c_{rkcd} \leq c_{agd}$ for $\eta\geq \eta_0$,	where $\eta_0\simeq 1.17$ is a moderate size constant.
\label{fig:figc1}}
	\end{figure}

\section{Perturbation of a quadratic objective function} \label{sec:PRKCD}

Proposition \ref{prp:eff} shows us that, \moda{for} strongly convex quadratic 
problems, one can recover the optimal convergence rate of the 
conjugate gradient for large values of the damping parameter $\eta$.
Here we discuss a modification of Algorithm \ref{alg:RKCD} to a 
specific class of nonlinear problems for which we can prove the 
same convergence rate as in the quadratic case. More precisely, we  consider 
\begin{equation} \label{eq:comp}
f(x)=\tfrac{1}{2}x^{T}Ax+g(x),
\end{equation}
where $A$ is \moda{a} positive definite matrix for which we know \moda{(bounds on)} the smallest 
and largest eigenvalue,  and $g$ is a $\beta$-smooth convex 
function \cite{N14}. 
Inspired by \cite{Zbi11}, we consider the  modification\footnote{In the case where $g(x)=0$ this 
algorithm coincides with  Algorithm \ref{alg:RKCD}  for $\nabla f(x)=Ax$.} of Algorithm 
\ref{alg:RKCD} specifically designed for the minimization of \moda{composite} functions of the form 
\eqref{eq:comp}. We call this method 
\modgg{the partitioned Runge-Kutta Chebyshev descent method (PRKCD)} 
and show in Proposition~\ref{thm:partitioned} (proved in the appendix) that it matches the rate given by the analysis of 
quadratic problems. 
\modgg{PRKCD is derived as a variant of the RKCD method applied to the target function
\eqref{eq:comp}, where the nonquadratic term is replaced by the linearization 
$g(x)\simeq g(x_n) + (x-x_n)^T\nabla g(x_n)$, as described in Algorithm~\ref{alg:CRKCD}.
In practice, PRKCD is implemented almost identically to the RKCD method, except that the gradient terms $\nabla g$ in $\nabla f(x)=Ax+\nabla g(x)$ are evaluated at $x_n$ in all the intermediate steps $j=1,\ldots,s$ described in Algorithm~\ref{alg:RKCD}. Hence,}
PRKCD is a numerically stable implementation of the update
\begin{align} \label{eq:comp_scheme}
x_{n+1} & = R_{s}(-Ah)x_{n}-B_s(-Ah)\nabla g(x_{n}), \\ B_s(z)  &=\frac{1-R_{s}(z)}{z}, \nonumber
\end{align}
where $R_{s}(z)$ is the stability function given by~\eqref{eq:exp_stab}. It is worth noting that this method has only one evaluation of $\nabla g$ per step of the algorithm, which can be advantageous if the evaluations of $\nabla g$ are costly. 

\begin{center}
\begin{algorithm}
\caption{\modgg{Partitioned Runge-Kutta Chebyshev descent method (PRKCD) for minimizing}  ~\eqref{eq:comp} \label{alg:CRKCD}}

\textbf{Input:}
The gradient $\nabla g$ of a differentiable function $g:\mathbb{R}^d\rightarrow\mathbb{R}$, and the matrix $A\in \mathbb{R}^{d\times d}$. Damping $\eta$ (e.g., $\eta=1.17$). Lower and upper bounds $\ell$ and $L$ for eigenvalues of $A$.  Initialisation $x_0\in\mathbb{R}^d$.
\vspace{0.1cm}

\modgg{Apply Algorithm \ref{alg:RKCD} to the function 
\begin{equation} \label{eq:complin}
f(x)=\tfrac{1}{2}x^{T}Ax+g(x_n)+(x-x_n)^T\nabla g(x_n),
\end{equation}
with gradient $\nabla f(x)=Ax+\nabla g(x_n).$}
%
%
\end{algorithm}
\end{center}
\begin{prop} \label{thm:partitioned}
Let $f$ be given by (\ref{eq:comp}) where $A\in\mathbb{R}^{d\times d}$ is a positive definite  matrix with largest and smallest eigenvalues $L$ and $\ell > 0$, respectively, and condition number $\kappa=L/\ell$. Let $\gamma >0$ and $g$ be a $\beta$-smooth convex function for which $0<\beta<C(\eta) \gamma \ell$ where \modb{$C(\eta)= \omega_{1}  \alpha_{s}(\eta) / (\omega_{0}-1)$} and the parameters of PRKCD are chosen according to conditions (\ref{eq:condition_2a}). 
Then the iterates of PRKCD  satisfy
\begin{equation} \label{eq:decay_linearized}
f(x_{n+1})-f(x_{*}) \leq (1+\gamma)^{2}\alpha^{2}_{s}(\eta) (f(x_{n})-f(x_{*}))
\end{equation}
with $x_*$ the minimizer of $f$.
\end{prop}	
Roughly speaking, Proposition~\ref{thm:partitioned} states that the effective convergence rate of PRKCD matches that of RKCD in~\eqref{eq:decayf} up to a factor of $(1+\gamma)^2$, as long as 
\begin{equation}\label{eq:condconv}
(1+\gamma) \alpha_s(\eta) < 1
\end{equation}
 to guarantee \moda{the}  convergence \modgg{of Algorithm \ref{alg:CRKCD}}. 
Since $\lim_{s\rightarrow \infty}(1+\gamma)^{2/s}=1$, the effective rate in
\eqref{eq:decay_linearized} is equivalent to $c_{rkcd}(\kappa)$ in the limit of a large condition number $\kappa$.
Indeed, the numerical evidence in Figure \ref{fig:figc1} 
suggests that PRKCD remains efficient for moderate values of the damping parameter $\eta$. In particular, for $\moda{\eta = } \eta_{0}=1.17$, we have \modgg{that $c_{rkcd}\simeq c_{agd}=(1-2/\sqrt{3\kappa +1})^2$ and} $C(\eta_{0})\simeq 0.59$ when $\kappa \gg1$. 
\modgg{Using \eqref{eq:condition_2a},\eqref{eq:rate} yields that $\alpha_s(\eta_0)$ is close to $\lim_{\kappa \rightarrow \infty} c_{agd}^{s/2} \simeq 0.413$, and the convergence condition \eqref{eq:condconv} holds for $\gamma< 1/\alpha_s(\eta_0)-1$ and $\beta/\ell < C(\eta_0)\gamma \simeq 0.83$ for $\kappa \gg1$, where $\beta$ is the smoothness parameter of $g$ and $\ell$ the smallest eigenvalue of $A$.}

\modgg{
\begin{remark} \label{rem:g1g2}
For the case of a stiff objective function $g_1$ perturbed by a nonstiff and possibly costly nonquadratic function $g_2$, the PRKCD method can be generalized as follows to minimize the function
\begin{equation} \label{eq:g1g2}
f(x)=g_1(x)+g_2(x).
\end{equation}
 Similar to Algorithm \ref{alg:CRKCD}, one can simply apply Algorithm \ref{alg:RKCD} to the modified objective function where $g_2(x)$ is replaced by $g_2(x_n)+(x-x_n)^T\nabla g_2(x_n)$ in \eqref{eq:g1g2}.
Note that the corresponding method coincides with Algorithm \ref{alg:CRKCD} in the case of a quadratic function~$g_1(x)=\tfrac{1}{2}x^{T}Ax$ perturbed by $g_2(x)=g(x)$.
\end{remark}
}

\section{Numerical examples \label{sec:numerics}}
In this section we illustrate the performance of of RKCD (Algorithm~\ref{alg:RKCD}) and PRKCD (Algorithm~\ref{alg:CRKCD}) for solving \moda{problem} ~\eqref{eq:main} for a number of different test problems. In all the test problems, we  \moda{will} compare our method with optimally-tuned gradient descent (GD), given by 
\[
x_{k+1}=x_{k}-\frac{2}{\ell+L} \nabla f(x_{k}),
\]
as well as \moda{the} accelerated gradient descent  (AGD), given \moda{in \cite{N14} as}
\begin{align*}
x_{k+1} &= y_{k}-\frac{1}{L} \nabla f(y_{k}), \\
y_{k+1} &=x_{k+1} +\frac{\sqrt{L}-\sqrt{\ell}}{\sqrt{L}+\sqrt{\ell}}(x_{k+1}-x_{k}).
\end{align*}

\modb{Two of our examples are for quadratic stiff problems, either pure  or perturbed by a non-stiff term, while the other test problems are for stiff non-quadratic problems.  This is only to verify the theoretical analyses of Propositions~\ref{prp:basic}--\ref{thm:partitioned}. We do not advocate RKCD and PRKCD for such problems since exponential integrators with (rational) Krylov subspace techniques might be more appropriate then; see, e.g.,~\cite{Hochbruck:2010,druskin2010adaptive}. However, the experiments below will show that our methods also perform well for minimizing non-quadratic strongly convex functions, which our theoretical results do not cover.}

\paragraph{Strongly convex quadratic programming:} We consider the quadratic 
function  $f(x) = \tfrac{1}{2} x^T A x - x^T b$, with $A \in \mathbb{R}^{n \times n}$ 
a random matrix drawn from the Wishart distribution\moda{, namely,} $A \sim \frac{1}{m}W_{n}(I_{n},m)$ \moda{with} $n=4800, m=5000$, and \moda{the entries of} $b \in \mathbb{R}^{n}$ \moda{are independently drawn from the standard Gaussian distribution} $\mathcal{N}(0,1)$. For this matrix, \moda{with high probability}, we have the following  estimates \moda{from \cite{vershynin2010introduction}} for its largest and smallest eigenvalues
\[
L=\left(1+\sqrt{\frac{n}{m}} \right)^{2},  \quad \ell=\left(1-\sqrt{\frac{n}{m}} \right)^{2}
\]
giving $\kappa \approx 10^4$. Since this problem is quadratic we \moda{will} also solve it using the conjugate gradient method (CG). 

We plot $f(x_k)-f(x_{*})$ for  the RKCD, GD, AGD and CG in Figure \ref{fig:quad}. As predicted by \moda{Proposition \ref{prp:eff}}, as the damping parameter $\eta$ for RKCD increases, the behaviour of \moda{RKCD approaches that}  of CG. Furthermore, \moda{even for the modest damping parameter of} $\eta \approx 1.17$, RKCD \moda{is comparable} to AGD.

\paragraph{Regularised logistic regression:} We \moda{now} study our first example not 
covered by our theoretical results. \moda{Consider }the objective function 
$$
 f(x) = \sum_{i=1}^m \log(1 + \exp(-y_i \xi_i^T x)) + \tfrac{\tau}{2} \|x\|^2_2,
$$
where $\Xi = \begin{bmatrix} \xi_1 & \cdots & \xi_m \end{bmatrix}^T \in \mathbb{R}^{m \times d}$ is the design matrix \moda{with the columns $\{\xi_i\}_i$,} and $y \in \{ -1, 1 \}^d$. \moda{Note that} $f \in \mathcal{F}_{\ell,L}$ with $\ell=\tau$ and $L=\tau + \|\Xi\|_2^2/4$. As in~\cite{Scieur:2016}, we used the Madelon UCI dataset with $d = 500$, $m = 2000$, and $\tau =10^2$. This \moda{results in} a very poorly conditioned problem with $\kappa = L/\moda{l} \approx 10^9$.

We plot  $f(x_k)-f(x_{*})$ for  the RKCD, GD, and AGD in Figure \ref{fig:logistic}. We again observe that, as the damping parameter increases, the 
convergence \moda{rate of RKCD} improves, \moda{thus requiring} half \moda{of the} number of evaluations \moda{needed} for the RKCD to achieve the same level of accuracy \moda{of} $10^{-5}$, \moda{compared to} AGD. Furthermore, for this problem, RKCD is \moda{comparable} to  AGD \moda{even for the modest damping parameter of} $\eta \approx 1.17$.

\paragraph{Regression with (smoothed) elastic net regularisation:} We now study a regression problem with (smoothed) elastic net regularisation. In particular, \moda{in} the case the objective function is of the form
$
 f(x) = \tfrac{1}{2} \| Ax-b \|_2^2 + \lambda L_\tau(\|x\|_1) + \tfrac{\ell}{2} \|x\|^2_2,
 $
 where $L_\tau(t)$ is the standard Huber loss function with parameter $\tau$ to smooth $|t|$; see~\cite{zou2005regularization}. We used $A \in \mathbb{R}^{m \times d}$ a random matrix drawn from the Gaussian distribution scaled by $1/\sqrt{d}$, $d=3000$, $m=900$, $\lambda=0.2$, $\tau=10^{-3}$, $\moda{l}=10^{-2}$. The objective function is in $\mathcal{F}_{\ell,L}$ with $L \approx (1+\sqrt{m/d})^2 + \lambda/\tau + \ell$. The condition number is $\kappa \approx 10^{4}$.

 We plot  $f(x_k)-f(x_{*})$ for  the RKCD, GD, and AGD in Figure \ref{fig:el_net}. The behaviour of all the different methods is very similar to the one for the regularised logistic regression. 

\begin{figure}[htb]
\center
\includegraphics[scale=1.0]{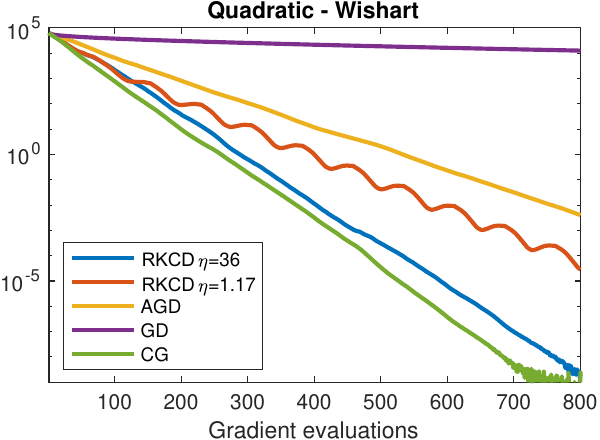}
\caption{Error in function value $f(x_k) - f(x_*)$ for the strongly convex quadratic programming problem.}
\label{fig:quad}
\end{figure}

\begin{figure}[htb]
\center
\includegraphics[scale=1.0]{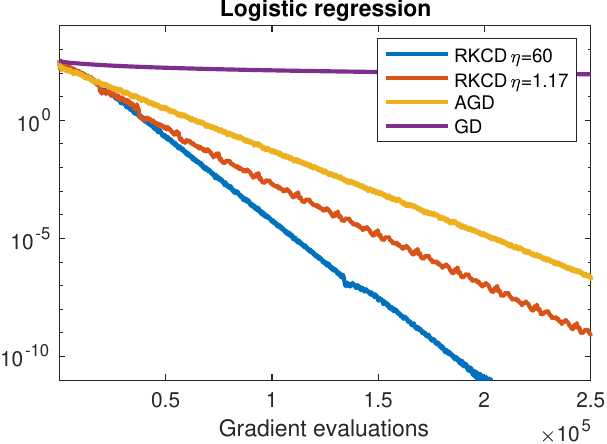}
\caption{Error in function value $f(x_k) - f(x_*)$ for the regularised logistic regresion problem.}
\label{fig:logistic}
\end{figure}

 \begin{figure}[htb]
\center
\includegraphics[scale=1.0]{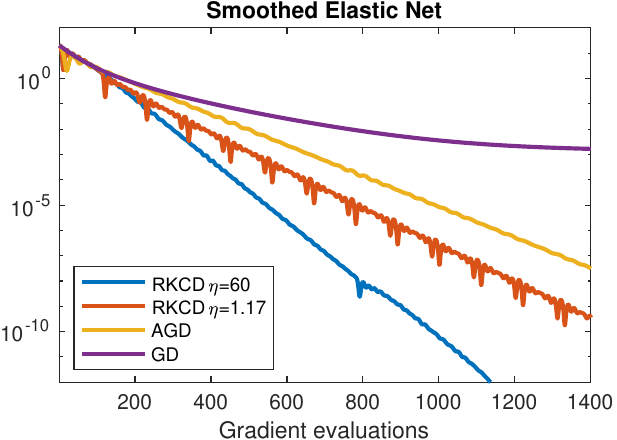}
\caption{Error in function value $f(x_k) - f(x_*)$ for the regression with (smoothed) elastic net regularisation.}
\label{fig:el_net}
\end{figure}

\paragraph{Nonlinear elliptic PDE: } We consider the  one-dimensional integro-differential PDE 
\begin{equation}\label{eq:pde}
\modgg{\frac{\partial}{\partial x} \left(\exp\big(-\alpha u(x)\big) \frac{\partial u(x)}{\partial x}\right)} 
= \int_0^1 \frac{u^{4}(s)}{(1+|x-s|)^2} dx, \quad u(0)=1, \ u(1)=0
\end{equation}
on \moda{the interval} $0\leq x\leq 1$, \modgg{where $\alpha$ is a fixed parameter.
We first consider the semilinear case $\alpha=0$, where the left-hand side of \eqref{eq:pde} reduces 
to the Laplacian $\frac{\partial^{2} u(x)}{\partial x^2}$, which yields a model} describing
the stationary variant of a temperature profile of air near the ground \cite{VaV92}. Using a finite difference approximation $u(i\Delta x) \simeq U_i$ \moda{for} $i=1,\ldots d$ 
on a spatial mesh with size $\Delta x=1/(d+1)$,  
and using the trapezoidal quadrature for the integral, we obtain 
a problem of the form~\eqref{eq:comp} where
$A$ is the usual tridiagonal discrete Laplace matrix of size $d\times d$ \modg{with condition number $\kappa \simeq 4\pi^{-2}d^2$ (using $\lambda_{\max} \sim 4\Delta x^{-2}$ and $\lambda_{\min} \rightarrow \pi^2$ as $\Delta x\rightarrow 0$)}, 
and the entries of \moda{the gradient vector} $\nabla g(U) \in\mathbb{R}^d$ are given by
$$
\frac{\partial g(U)}{\partial U_i} =
\frac{\Delta x}{2(1+i\Delta x)^2}
+ \sum_{j=1}^d\frac{\Delta x U_j^4}{(1+\Delta x|i-j|)^2}.
$$
\moda{We observe that,} when the dimension $d$ is large, calculating $\nabla g(x)$ can become very expensive as one needs to calculate a sum over all $j$'s, \moda{which makes PRKCD a particularly attractive option here.}  We plot $f(x_k)-f(x_{*})$ for  the RKCD, PRKCD and GD  in Figure \ref{fig:nl_PDE} for \modgg{$d=200$.
We use the initialization corresponding to the simple function $u(x)=1-x$ which satisfies the boundary conditions in \eqref{eq:pde}.}

\modgg{
In the semilinear case ($\alpha=0$), we see in Figure \ref{fig:nl_PDE} (left) that }
the PRKCD performs similarly to RKCD for the sets of parameters
\modgg{$\eta=10,s=288$ and $\eta=1.17,s=99$}, \modg{where $s$ denotes the number of evaluations per step of 
$\nabla f(x)=Ax+\nabla g(x)$ for RKCD, and respectively the number of matrix-vector products $Ax_n^{j-1}$ for PRKCD (recall that PRKCD needs a single evaluation of the gradient $\nabla g$ per step)}, which corroborates Proposition~ \ref{thm:partitioned}. 
\modgg{For the case of a nonlinear diffusion (with $\alpha=1$ in \eqref{eq:pde}), we consider the natural generalization of the PRKCD method described in Remark \ref{rem:g1g2}. For the discrete nonlinear diffusion term evaluated at point $x_i=i\Delta t$, we use the natural finite difference formula
$$
\frac{\partial}{\partial x} \left(\exp\big(-\alpha u(x_i)\big) \frac{\partial u(x_i)}{\partial x}\right)
\simeq e^{-\alpha U_i}\frac{U_{i+1}-U_i}{\Delta x^2} - 
e^{-\alpha U_{i-1}} \frac{U_{i}-U_{i-1}}{\Delta x^2},
$$
and consider in the algorithms the bounds $\ell = \pi^2e^{-\alpha}$ and $L=4\Delta x^{-2}$ for the spectrum of $\nabla^2 f(x)$.
In Figure \ref{fig:nl_PDE} (right) we observe again excellent performances of the RKCD and PRKCD methods, although such a nonquadratic problem associated to the nonlinear PDE \eqref{eq:pde} is not covered by Proposition~\ref{thm:partitioned}.
}
\begin{figure}[htb]
\centering
    \begin{subfigure}[b]{0.45\textwidth}
        \includegraphics[width=\textwidth]{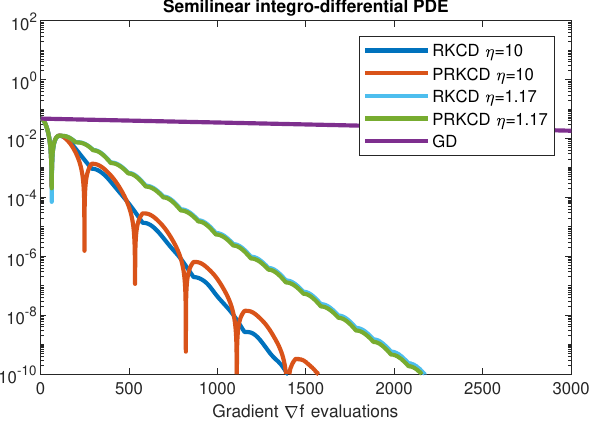}
        \caption{Semilinear PDE ($\alpha=0$)}
        \label{fig:semilinear}
    \end{subfigure}
    \hfill
    \begin{subfigure}[b]{0.45\textwidth}
        \includegraphics[width=\textwidth]{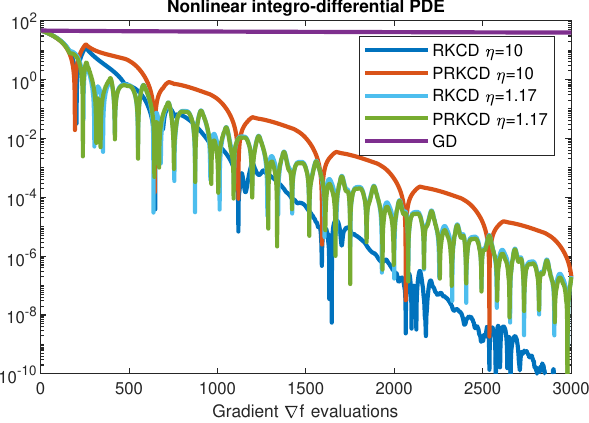}
        \caption{\moda{Semilinear PDE ($\alpha=1$)}}
        \label{fig:nonlinear}
    \end{subfigure}
    ~ 
      \caption{\moda{Error in function value $f(x_k) - f(x_*)$ for the nonlinear PDE problem~\eqref{eq:pde}}.}\label{fig:nl_PDE}
\end{figure}

\paragraph{Smoothed total variation denoising:} Here we consider the problem of image denoising using a smoothed total variation regularisation. In particular the objective function is of the form
\[
f(x)=\frac{1}{2} ||x-y||^{2}+\lambda J_{\epsilon}(x)
\] 
where $y$ is the noisy image\footnote{The original image and the noisy version can be seen in Figure~\ref{fig:images}.} and $J_{\epsilon}(x)$ is is a smoothed total variation of the image.

 \begin{figure}
    \centering
    \begin{subfigure}[b]{0.4\textwidth}
        \includegraphics[width=\textwidth]{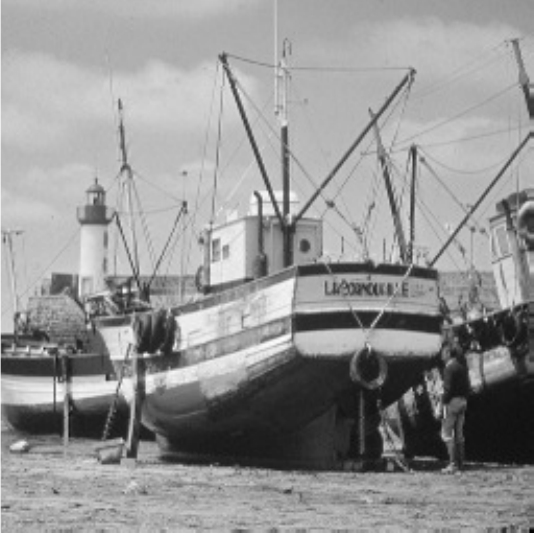}
        \caption{Original image}
        \label{fig:original}
    \end{subfigure}
    \hfill
    \begin{subfigure}[b]{0.4\textwidth}
        \includegraphics[width=\textwidth]{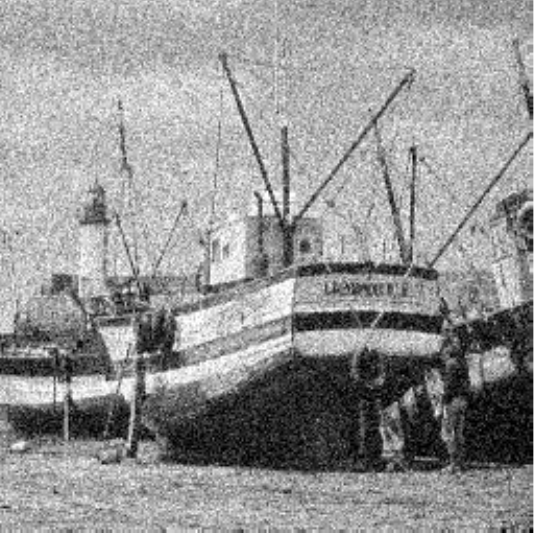}
        \caption{Noisy image}
        \label{fig:noisy}
    \end{subfigure}
    ~ 
      \caption{\moda{Images used for the smoothed total variation denoising.}}\label{fig:images}
\end{figure}
In particular,
\[
J_{\epsilon}(x) =\sum_{i} ||(Gx)_{i}||_{\epsilon}
\]
where $(Gx)_{i} \in \mathbb{R}^{2}$ is an approximation of the gradient of $x$ at  pixel $i$ and for $u \in \mathbb{R}^{2}$ and 
\[
||u||_{\epsilon}=\sqrt{\epsilon^{2}+||u||^{2}}
\] 
a smoothing of the $L^{2}$ norm in $\mathbb{R}^{2}$.

The objective function is in $\mathcal{F}_{\ell,L}$ with $\ell=2$ and $L=2\left(1+\frac{4\lambda}{\epsilon} \right)$.  We choose $\lambda=6 \cdot 10^{-2}, \epsilon=10^{-4}$ and hence the conditioning number is $\kappa \approx 2.4 \times 10^{4}$. We plot $f(x_k)-f(x_{*})$ for  the RKCD,  GD and AGD in Figure \ref{fig:image}a as well as the denoised image in Figure \ref{fig:image}b. 

\begin{figure}[h]
    \centering
    \begin{subfigure}[b]{0.49\textwidth}
        \includegraphics[width=\textwidth]{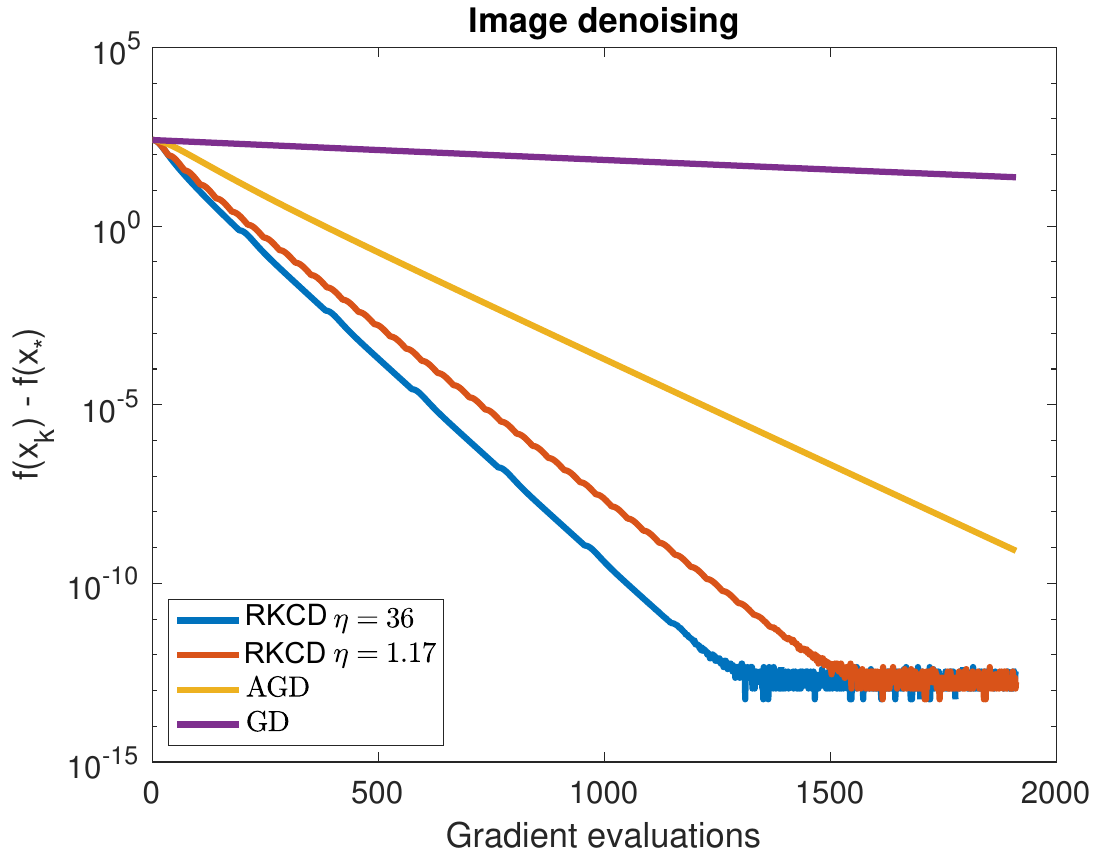}
        \caption{\moda{Error in function values.}}
        \label{fig:decay}
    \end{subfigure}
    \hfill
    \begin{subfigure}[b]{0.4\textwidth}
        \includegraphics[width=\textwidth]{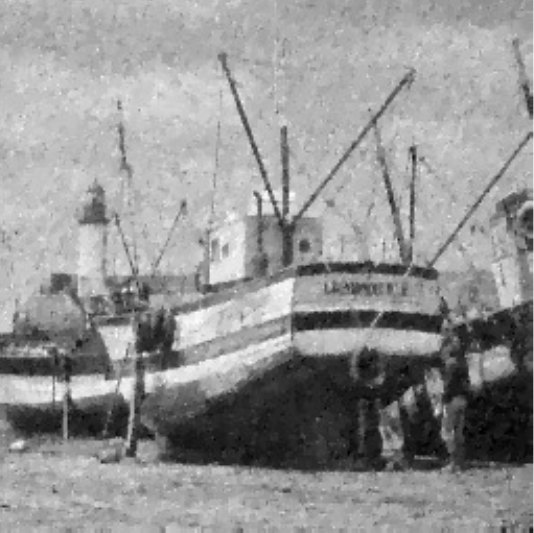}
        \caption{Denoised image}
        \label{fig:denoise}
    \end{subfigure}
    ~ 
      \caption{\moda{Total variation image denoising example.}}\label{fig:image}
\end{figure}

\section{Conclusion}
In this paper, using ideas from numerical analysis of ODEs, we 
introduced a new class of optimisation algorithms for strongly convex 
functions based on explicit stabilised methods. With some care, these new methods RKCD and PRKCD are as easy to 
implement as SD but require in addition a lower bound $\ell$ on the smallest eigenvalue. They were shown to match the optimal 
convergence rates of first order methods for certain subclasses of 
$\mathcal{F}_{\ell,L}$. 

Our numerical 
experiments illustrate that this might be the case for all functions in  $
\mathcal{F}_{\ell,L}$, and proving this is the subject of our current 
research efforts. In addition, there is a number of different interesting 
research avenues for this class of methods, including adjusting them 
to convex optimisation problems with $\ell=0$, as well as for 
adaptively choosing the time-step $h$ using local information to 
optimise their performance further. Furthermore,  their adaptation to 
stochastic optimisation problems, where one replaces the full 
gradient of the function $f$ by a noisy but cheaper version of it, 
is another interesting but challenging direction we aim to investigate further.  

\section*{Acknowledgements}
We thank the referees for their constructive remarks when reviewing this paper. \modgg{This work was partially supported by The Alan Turing Institute by a small  grant award SG020. The research of K.C.Z was
partially supported by the Alan Turing Institute under the EPSRC grant EP/N510129/1. The research of B.V. and G.V.~was partially supported by grant 200020\_178752 of the SNSF (Swiss National Science Foundation). 
The research of G.V.~was partially supported by grant 200020\_184614 of the SNSF.}

\appendix
\numberwithin{equation}{section}
\section{Proof of the main results}
In this section we will discuss the proofs of the main results in the paper
\begin{proof}[Proof of Proposition~\ref{prp:basic}]
 We start our proof by noticing that the gradient flow~\eqref{eq:gradient} in the case of the quadratic function~\eqref{eq:quad} becomes
\[
\frac{dx}{dt}=-Ax+b.
\] 
In addition since $A$ is positive definite and symmetric, there 
exists an orthogonal matrix $V$ such that 
\[
A=VDV^{-1}, \quad D=\text{diag}(\lambda_{1}, \cdots, \lambda_{d}), \quad \lambda_{1} \leq \cdots \leq \lambda_{d}.
\]
If we now make the change of variables $y=V^{-1}x-D^{-1}V^{-1}b$, we obtain the following equation
\begin{equation} \label{eq:gradient_independent}
\frac{dy}{dt}=-Dy.
\end{equation}
Hence, in this coordinate system each coordinate is 
independent of the other, while the objective function  can be 
written as 
\[
f(y)-f(y^{(*)})=\frac{1}{2}\sum_{i=1}^{d}\lambda_{i}y^{2}_{i},  \quad y^{(*)}=0.
\]
We can now write equation~\eqref{eq:gradient_independent} in 
the vector form as
\[
\frac{dy_{i}}{dt}=-\lambda_{i}y_{i},
\]
and hence each coordinate satisfies independently the simple quadratic~\eqref{eq:quad} and the application of Runge--Kutta method with stability function $R(z)$ gives $y^{(n+1)}=R(-h\lambda_{i}) y^{(n)}$ (where $y^{(n)}$ is the $n$th iterate, that is, $y^{(n)} = V^{-1}x_n$). Hence
\begin{eqnarray*}
f(y^{(n+1)})-f(y^{(*)}) &=& \sum_{i=1}^{d}\lambda_{i}\left[R(-h\lambda_{i})y^{(n)}_{i}\right]^{2} \\
		&\leq&	\max_{1 \leq i\leq d}R^{2}(-h \lambda_{i}) \, ( f( y^{(n)})- f(y^{(*)})),
\end{eqnarray*}
which completes the proof.
\end{proof}

\begin{proof}[Proof of Proposition \ref{prp:eff}]
Using~\eqref{eq:alpha} and properties of Chebyshev polynomials we have
\begin{equation*} 
\alpha_{s}(\eta)=\left [\cosh \left( s\operatorname {arcosh}\left(1+\frac{\eta}{s^{2}} \right)\right) \right]^{-1}.
\end{equation*}
Using~\eqref{eq:condition_2a}  and the estimate $\cosh(s x)^{-2/s} \rightarrow e^{-2x}$ for $s\rightarrow \infty$, we deduce 
\begin{align*}
\lim_{\eta \rightarrow \infty}c_{rkcd}(\kappa)&=e^{-2\operatorname {arcosh}\left(1+\frac{2}{\kappa} \right)}\\
&=\left( \frac{\sqrt{\kappa}-1}{\sqrt{\kappa}+1} \right)^{2}+O(\kappa^{-3/2}). 
\end{align*}
\end{proof}

\begin{proof}[Proof of Proposition~\ref{thm:partitioned}]
Our starting point is
equation~\eqref{eq:comp_scheme}. In particular, we will show that if we choose our parameters suitably, then this scheme converges with the rate predicted by the theorem, and we will then show how Algorithm~\ref{alg:CRKCD} corresponds to an implementation of this scheme. 

We start the proof by noting that since $f$ in~\eqref{eq:comp} is strongly convex there exists a  unique minimizer $x_{*}$ satisfying 
\begin{equation} \label{eq:step1}
Ax_{*}+\nabla g(x_{*})=0.
\end{equation}
Thus 
\begin{align*}
x_{n+1}-x_{*} &=R_{s}(-Ah)x_{n}-h B_s(-Ah) \nabla g(x_{n})-x_{*} \\
		     &=R_{s}(-Ah)(x_{n}-x_{*}) \\
         &\quad-hB_s(-Ah) (\nabla g(x_{n})-\nabla g(x_{*}))\\
				&\quad+	(R_{s}(-Ah)-I+hB_s(-Ah)A)x_{*}
\end{align*}
where in the above identity we have used~\eqref{eq:step1} multiplied on the left by $B_s(-Ah)$. Now, by definition of the matrix $B$ we have $R_{s}(-Ah)-I+hB_s(-Ah)A=0$, and we obtain
\begin{align*}
x_{n+1}-x_{*}&=R_{s}(-Ah) (x_{n}-x_{*})\\
 &\quad -hB_s(-Ah) (\nabla g(x_{n})-\nabla g(x_{*}))
\end{align*}
and hence 
\[
||x_{n+1}-x_{*}|| \leq \left(||R_{s}(-Ah)||+h ||B_s(-Ah)|| \beta \right) ||x_{n}-x_{*}||.
\]
We now know from the analysis in the main text that if $h,s$ are chosen according to~\eqref{eq:condition_2a}, then 
\[
||R(-Ah) || \modg{\leq} \alpha_{s}(\eta)
\]
and using the fact that $||B_s(-Ah)|| \leq 1$, \modg{which is a consequence of Lemma \ref{lemmaB} below,} we see that 
\[
||x_{n+1}-x_{*}|| \leq \left[\alpha_{s}(\eta)+\beta \left(\frac{\omega_{0}-1}{\omega_{1}\ell}  \right) \right] ||x_{n}-x_{*}||
\]
and hence
\[
||x_{n+1}-x_{*} || \leq (1+\gamma) \alpha_{s}(\eta) ||x_{n}-x_{*}||
\]
for  $$0< \beta < \beta_{*}=\gamma\mu C(\eta)$$ 
where we define
\begin{equation*} 
C(\eta) = \frac{ \omega_{1}  \alpha_{s}(\eta)}{\omega_{0}-1}.
\end{equation*}
We have thus proved Proposition~\ref{thm:partitioned} for a numerical scheme of the form~\eqref{eq:comp_scheme}. 
It remains to prove that Algorithm \ref{alg:CRKCD} is a  scheme equivalent to~\eqref{eq:comp_scheme}. 
Indeed, the following identity can be proved by induction on $j=2,\ldots,s$, using $T_j(x)=2xT_{j-1}(x)-T_{j-2}(x)$,
$$
x_n^j = R_{s,j}(-hA) x_n -B_{s,j}(-hA)\nabla g(x_n)
$$
where $R_{s,j}(x) = {T_j(\omega_0+\omega_1 x)}/{T_j(\omega_0)}$, $B_{s,j}(x)=(I-R_{s,j}(x))/x$, and we obtain the result~\eqref{eq:comp_scheme} by taking $j=s$.
\end{proof}
\modg{
\begin{lemma} \label{lemmaB}
Consider the polynomial $B(\xi)$ defined in
\eqref{eq:comp_scheme}. 
Then,
$
\max_{\xi \in [-L_{s,\eta},0]} |B(\xi)| = 1.
$
\end{lemma}
\begin{proof}
By the Lagrange theorem, for $\xi$ in the interval $(-L_{s,\eta},0)$, there
exists $z$ in the same interval such that $-B(\xi)=(R_s(z)-1)/z=R_s'(z)$, and using the definition of $R_s(z)$ in \eqref{eq:exp_stab}, we have for 
$x=\omega_0 +\omega_1 z$,
$$
R_s'(z) = \frac{T_s'(x)}{T_s'(\omega_0)}.
$$
Recalling $x\in [-1,\omega_0]$, 
for the first case where $x \in [1,\omega_0]$, 
we use the fact that $T_s(x)$ and all its derivatives are increasing functions of $x\geq 1$ (this is because by the Rolle theorem all the roots 
of the Chebyshev polynomials $T_s(x)$ and their derivatives are in the interval $[-1,1]$). Hence $|T_s'(x)| \leq |T_s'(\omega_0)|$, which yields
$|R_s'(z)|\leq 1$.
For the second case where 
$x \in [-1,1]$, we have
$x=\cos (\theta)$ for some $\theta$, and we obtain
$$
s^{-2} |T_s'(x)| = s^{-2} |T_s'(\cos(\theta))|
= \frac{|\sin(s \theta)|}{s|\sin \theta|} \leq 1,
$$
where the latter bound can be obtain by an elementary study of the function 
$\sin(s \theta)/(s\sin \theta)$,
while $s^{-2} T_s'(\omega_0) \geq s^{-2} T_s'(1) =1$ is an increasing function of $\eta$, hence $|R_s'(z)|\leq 1$, and this concludes the proof of Lemma~\ref{lemmaB}.
\end{proof}
}


\begin{thebibliography}{10}

\bibitem{N14}
Yurii Nesterov.
\newblock {\em Introductory Lectures on Convex Optimization: A Basic Course}.
\newblock Springer Publishing Company, Incorporated, 1 edition, 2014.

\bibitem{SRB17}
Damien Scieur, Vincent Roulet, Francis~R. Bach, and Alexandre d'Aspremont.
\newblock Integration methods and optimization algorithms.
\newblock In {\em Advances in Neural Information Processing Systems 30: Annual
  Conference on Neural Information Processing Systems 2017, 4-9 December 2017,
  Long Beach, CA, {USA}}, pages 1109--1118, 2017.

\bibitem{SBC16}
Weijie Su, Stephen Boyd, and Emmanuel~J. Cand{{\`e}}s.
\newblock A differential equation for modeling nesterov's accelerated gradient
  method: Theory and insights.
\newblock {\em Journal of Machine Learning Research}, 17(153):1--43, 2016.

\bibitem{WWJ16}
Andre Wibisono, Ashia~C. Wilson, and Michael~I. Jordan.
\newblock A variational perspective on accelerated methods in optimization.
\newblock {\em Proceedings of the National Academy of Sciences},
  113(47):E7351--E7358, 2016.

\bibitem{WRJ16}
Ashia~C. Wilson, Benjamin Recht, and Michael~I. Jordan.
\newblock A {L}yapunov analysis of momentum methods in optimization.
\newblock {\em arXiv:1611.02635}, 2016.

\bibitem{BJW18}
Michael Betancourt, Michael~I. Jordan, and Ashia~C. Wilson.
\newblock On symplectic optimization.
\newblock {\em arXiv:1802.03653}, 2018.

\bibitem{HaW96}
Ernst Hairer and Gerhard Wanner.
\newblock {\em Solving ordinary differential equations II. Stiff and
  differential-algebraic problems}.
\newblock Springer-Verlag, Berlin and Heidelberg, 1996.

\bibitem{NB14}
Neal Parikh and Stephen Boyd.
\newblock Proximal algorithms.
\newblock {\em Found. Trends Optim.}, 1(3):127--239, January 2014.

\bibitem{SSV98}
Ben~P. Sommeijer, Laurence~F. Shampine, and Jan~G. Verwer.
\newblock {RKC}: an explicit solver for parabolic {PDEs}.
\newblock {\em J. Comput. Appl. Math.}, 88:316--326, 1998.

\bibitem{AbM01}
Assyr Abdulle and Alexei Medovikov.
\newblock Second order {C}hebyshev methods based on orthogonal polynomials.
\newblock {\em Numer. Math.}, 90(1):1--18, 2001.

\bibitem{Abd02}
Assyr Abdulle.
\newblock Fourth order {C}hebyshev methods with recurrence relation.
\newblock {\em SIAM J. Sci. Comput.}, 23(6):2041--2054, 2002.

\bibitem{Abd11}
Assyr Abdulle.
\newblock Explicit stabilized {R}unge-{K}utta methods.
\newblock In {\em Encyclopedia of Applied and Computational Mathematics}, pages
  460--468, Berlin, 2015. Björn Engquist (Ed.), Springer-Verlag.

\bibitem{AbC08}
Assyr Abdulle and Stephane Cirilli.
\newblock S-{ROCK}: {C}hebyshev methods for stiff stochastic differential
  equations.
\newblock {\em SIAM J. Sci. Comput.}, 30(2):997--1014, 2008.

\bibitem{AAV18}
Assyr {Abdulle}, Ibrahim {Almuslimani}, and Gilles {Vilmart}.
\newblock {Optimal explicit stabilized integrator of weak order one for stiff
  and ergodic stochastic differential equations}.
\newblock {\em To appear in SIAM/ASA Journal on Uncertainty Quantification
  (JUQ)}, 2018.

\bibitem{VaS80}
Pieter~J. van~der Houwen and Ben~P. Sommeijer.
\newblock On the internal stability of explicit, {$m$}-stage {R}unge-{K}utta
  methods for large {$m$}-values.
\newblock {\em Z. Angew. Math. Mech.}, 60(10):479--485, 1980.

\bibitem{LRP16}
Laurent Lessard, Benjamin Recht, and Andrew Packard.
\newblock Analysis and design of optimization algorithms via integral quadratic
  constraints.
\newblock {\em SIAM Journal on Optimization}, 26(1):57--95, 2016.

\bibitem{Zbi11}
Christophe~J. Zbinden.
\newblock Partitioned {R}unge-{K}utta-{C}hebyshev methods for
  diffusion-advection-reaction problems.
\newblock {\em SIAM J. Sci. Comput.}, 33(4):1707--1725, 2011.

\bibitem{vershynin2010introduction}
Roman Vershynin.
\newblock Introduction to the non-asymptotic analysis of random matrices.
\newblock {\em arXiv preprint arXiv:1011.3027}, 2010.

\bibitem{Scieur:2016}
Damien Scieur, Alexandre {d'Aspremont}, and Francis Bach.
\newblock Regularized nonlinear acceleration.
\newblock In {\em NIPS'16 Proceedings of the 30th International Conference on
  Neural Information Processing Systems}, pages 712--720, 2016.

\bibitem{zou2005regularization}
Hui Zou and Trevor Hastie.
\newblock Regularization and variable selection via the elastic net.
\newblock {\em Journal of the Royal Statistical Society: Series B (Statistical
  Methodology)}, 67(2):301--320, 2005.

\bibitem{VaV92}
A.~S. Vasudeva~Murthy and Jan~G. Verwer.
\newblock Solving parabolic integro-differential equations by an explicit
  integration method.
\newblock {\em J. Comput. Appl. Math.}, 39(1):121--132, 1992.

\end{thebibliography}


\begin{thebibliography}{10}

\bibitem{Abd02}
A.~Abdulle.
\newblock Fourth order {C}hebyshev methods with recurrence relation.
\newblock {\em SIAM J. Sci. Comput.}, 23(6):2041--2054, 2002.

\bibitem{Abd11}
A.~Abdulle.
\newblock Explicit stabilized {R}unge-{K}utta methods.
\newblock In {\em Encyclopedia of Applied and Computational Mathematics}, pages
  460--468, Berlin, 2015. Bj{\"o}rn Engquist (Ed.), Springer-Verlag.

\bibitem{AAV18}
A.~Abdulle, I.~Almuslimani, and G.~Vilmart.
\newblock Optimal explicit stabilized integrator of weak order 1 for stiff and
  ergodic stochastic differential equations.
\newblock {\em SIAM/ASA Journal on Uncertainty Quantification}, 6(2):937--964,
  2018.

\bibitem{AbC08}
A.~Abdulle and S.~Cirilli.
\newblock S-{ROCK}: {C}hebyshev methods for stiff stochastic differential
  equations.
\newblock {\em SIAM J. Sci. Comput.}, 30(2):997--1014, 2008.

\bibitem{AbM01}
A.~Abdulle and A.~Medovikov.
\newblock Second order {C}hebyshev methods based on orthogonal polynomials.
\newblock {\em Numer. Math.}, 90(1):1--18, 2001.

\bibitem{BJW18}
M.~Betancourt, M.~I. Jordan, and A.~C. Wilson.
\newblock On symplectic optimization.
\newblock {\em arXiv:1802.03653}, 2018.

\bibitem{druskin2010adaptive}
V.~Druskin, C.~Lieberman, and M.~Zaslavsky.
\newblock On adaptive choice of shifts in rational krylov subspace reduction of
  evolutionary problems.
\newblock {\em SIAM Journal on Scientific Computing}, 32(5):2485--2496, 2010.

\bibitem{ERR18}
M.~J. Ehrhardt, E.~S. Riis, T.~Ringholm, and C.-B. Sch{\"o}nlieb.
\newblock A geometric integration approach to smooth optimisation: Foundations
  of the discrete gradient method.
\newblock {\em arXiv:1805.06444}, 2018.

\bibitem{HaL14}
E.~Hairer and C.~Lubich.
\newblock Energy-diminishing integration of gradient systems.
\newblock {\em IMA J. Numer. Anal.}, 34(2):452--461, 2014.

\bibitem{HaW96}
E.~Hairer and G.~Wanner.
\newblock {\em Solving ordinary differential equations II. Stiff and
  differential-algebraic problems}.
\newblock Springer-Verlag, Berlin and Heidelberg, 1996.

\bibitem{Hochbruck:2010}
M.~Hochbruck and A.~Ostermann.
\newblock Exponential integrators.
\newblock {\em Acta Numerica}, 19:209--286, 2010.

\bibitem{LRP16}
L.~Lessard, B.~Recht, and A.~Packard.
\newblock Analysis and design of optimization algorithms via integral quadratic
  constraints.
\newblock {\em SIAM Journal on Optimization}, 26(1):57--95, 2016.

\bibitem{N14}
Y.~Nesterov.
\newblock {\em Introductory Lectures on Convex Optimization: A Basic Course}.
\newblock Springer Publishing Company, Incorporated, 1 edition, 2014.

\bibitem{NB14}
N.~Parikh and S.~Boyd.
\newblock Proximal algorithms.
\newblock {\em Found. Trends Optim.}, 1(3):127--239, Jan. 2014.

\bibitem{Scieur:2016}
D.~Scieur, A.~{d'Aspremont}, and F.~Bach.
\newblock Regularized nonlinear acceleration.
\newblock In {\em NIPS'16 Proceedings of the 30th International Conference on
  Neural Information Processing Systems}, pages 712--720, 2016.

\bibitem{SRB17}
D.~Scieur, V.~Roulet, F.~R. Bach, and A.~d'Aspremont.
\newblock Integration methods and optimization algorithms.
\newblock In {\em Advances in Neural Information Processing Systems 30}, pages
  1109--1118, 2017.

\bibitem{shi2019acceleration}
B.~Shi, S.~S. Du, W.~Su, and M.~I. Jordan.
\newblock Acceleration via symplectic discretization of high-resolution
  differential equations.
\newblock In {\em Advances in Neural Information Processing Systems}, pages
  5745--5753, 2019.

\bibitem{SSV98}
B.~P. Sommeijer, L.~F. Shampine, and J.~G. Verwer.
\newblock {RKC}: an explicit solver for parabolic {PDEs}.
\newblock {\em J. Comput. Appl. Math.}, 88:316--326, 1998.

\bibitem{SBC16}
W.~Su, S.~Boyd, and E.~J. Cand{{\`e}}s.
\newblock A differential equation for modeling {N}esterov's accelerated
  gradient method: Theory and insights.
\newblock {\em Journal of Machine Learning Research}, 17(153):1--43, 2016.

\bibitem{VaS80}
P.~J. van~der Houwen and B.~P. Sommeijer.
\newblock On the internal stability of explicit, {$m$}-stage {R}unge-{K}utta
  methods for large {$m$}-values.
\newblock {\em Z. Angew. Math. Mech.}, 60(10):479--485, 1980.

\bibitem{VaV92}
A.~S. Vasudeva~Murthy and J.~G. Verwer.
\newblock Solving parabolic integro-differential equations by an explicit
  integration method.
\newblock {\em J. Comput. Appl. Math.}, 39(1):121--132, 1992.

\bibitem{vershynin2010introduction}
R.~Vershynin.
\newblock Introduction to the non-asymptotic analysis of random matrices.
\newblock {\em arXiv preprint arXiv:1011.3027}, 2010.

\bibitem{WWJ16}
A.~Wibisono, A.~C. Wilson, and M.~I. Jordan.
\newblock A variational perspective on accelerated methods in optimization.
\newblock {\em Proceedings of the National Academy of Sciences},
  113(47):E7351--E7358, 2016.

\bibitem{WMW19}
A.~C. Wilson, L.~Mackey, and A.~Wibisono.
\newblock Accelerating rescaled gradient descent: Fast optimization of smooth
  functions.
\newblock In {\em Advances in Neural Information Processing Systems 32}, pages
  13555--13565. 2019.

\bibitem{WRJ16}
A.~C. Wilson, B.~Recht, and M.~I. Jordan.
\newblock A {L}yapunov analysis of momentum methods in optimization.
\newblock {\em arXiv:1611.02635}, 2016.

\bibitem{Zbi11}
C.~J. Zbinden.
\newblock Partitioned {R}unge-{K}utta-{C}hebyshev methods for
  diffusion-advection-reaction problems.
\newblock {\em SIAM J. Sci. Comput.}, 33(4):1707--1725, 2011.

\bibitem{zhang2018direct}
J.~Zhang, A.~Mokhtari, S.~Sra, and A.~Jadbabaie.
\newblock Direct {R}unge-{K}utta discretization achieves acceleration.
\newblock In {\em Advances in neural information processing systems}, pages
  3900--3909, 2018.

\bibitem{zou2005regularization}
H.~Zou and T.~Hastie.
\newblock Regularization and variable selection via the elastic net.
\newblock {\em Journal of the Royal Statistical Society: Series B (Statistical
  Methodology)}, 67(2):301--320, 2005.

\end{thebibliography}


\end{document}